\documentclass[10pt, oneside, a4paper]{scrartcl}

\usepackage{amsmath,amssymb,amsthm}

\usepackage[T1]{fontenc}
\usepackage[utf8]{inputenc}
\usepackage[english]{babel}

\usepackage{fullpage}
\usepackage{microtype}
\usepackage{booktabs}
\usepackage[version=4]{mhchem}

\usepackage{csquotes} 
\usepackage[shortlabels]{enumitem}  
\usepackage{hyperref} 
\usepackage{xcolor}   
\usepackage{tikz}
\usepackage[normalem]{ulem}       

\usepackage{ifthen}

\usepackage{graphicx,float,subfigure}
\usepackage{capt-of}  
\graphicspath{{images/}}

\usepackage{acronym}

\usepackage[backend=biber,giveninits=true,style=numeric-comp,sorting=nyt,hyperref=true,url=false]{biblatex}
\newcommand*{\unit}[1]{\ensuremath{\mathrm{\,#1}}}

\usepackage{amssymb,amsmath,amsthm}
\usepackage{marvosym}
\usepackage{bbm}
\usepackage{dsfont}
\usepackage{array}
\usepackage{extpfeil}

\allowdisplaybreaks{}

\DeclarePairedDelimiter\norm{\|}{\|}

\DeclarePairedDelimiter\paren{(}{)}

\DeclarePairedDelimiter\inner{\langle}{\rangle}

\newcommand{\vfs}{\mathfrak{X}}     
\newcommand{\e}{\mathrm{e}}
\newcommand{\R}{\mathbb{R}}
\newcommand{\CC}{\mathbb{C}}

\newcommand{\diff}{\mathrm{d}}  
\newcommand{\im}{\mathrm{i}}  

\newcommand{\T}{\mathrm{T}}     
\newcommand{\NS}{\mathrm{N}}     
\newcommand{\UT}{\mathrm{UT}}     
\newcommand{\UN}{\mathrm{UN}}     
\newcommand{\sphere}{\mathbb{S}}  
\DeclareMathOperator{\spn}{span}

\DeclareMathOperator{\Lip}{Lip}

\newcommand{\deriv}[3][]{\frac{\diff^{#1}#2}{\diff #3^{#1}}}

\newcommand{\ambient}{M}
\newcommand{\ambdim}{m}
\newcommand{\nhim}{N}
\newcommand{\nhimdim}{n}
\newcommand{\flow}{\Phi}
\newcommand{\Def}[1]{{%
   \newcommand{\domain}{\mathcal{D}}%
   \ifthenelse{\equal{#1}{}}%
   {\domain}{\domain^{\left(#1\right)}}}}

\newcommand{\less}{\mbox{}\vspace{-\baselineskip}}

\newtheoremstyle{theorem}{}{}{\itshape}{}{\bfseries}{}{\newline}{}

\theoremstyle{theorem}
\newtheorem{theorem}{Theorem}
\newtheorem{lemma}[theorem]{Lemma}
\newtheorem{proposition}[theorem]{Proposition}
\newtheorem{corollary}[theorem]{Corollary}

\newtheoremstyle{example}{}{}{}{}{\bfseries}{}{\newline}{}

\theoremstyle{definition}
\newtheorem{definition}[theorem]{Definition}

\newtheorem{examples}[theorem]{Examples}

\title{Approximating normally attracting invariant manifolds using
trajectory-based optimization}
\author{Jörn Dietrich, Dirk Lebiedz}

\begin{document}

\maketitle

\begin{abstract}
  \textbf{Abstract:}
  The numerical simulation of chemically reactive flows is a major challenge due
  to stiffness and often high dimension of the corresponding kinetic differential
  equations.
  Manifold-based model reduction techniques address this problem by projecting
  the full phase space onto manifolds of `slow motion', which capture the
  system's long-term dynamics.
  In this article we study the trajectory-based optimization approach based on
  the original approach by Lebiedz (2004), which determines these manifolds in
  a variational formulation as minimizers of an appropriate functional.
  This article provides a rigorous explanation for its effectiveness, showing
  that it definitely approximates nonuniformly normally attracting orbits.
  It also outlines how the method can be utilized to approximate nonuniformly
  normally attracting invariant manifolds of higher dimension.
  Throughout the article we use a coordinate-free formulation on a Riemannian
  manifold. This is especially useful for systems subject to nonlinear
  constraints, e.g., adiabatic conditions in isolated chemical systems.
\end{abstract}

\section{Introduction}
The qualitative fundamental topological structure of a smooth dynamical system
is encoded in its invariant sets and manifolds.
When it comes to the long-term behavior of dissipative systems those
invariant manifolds that are relatively slow compared to nearby orbits are of
particular interest.
These manifolds usually offer an effective model reduction approach reducing
both dimension and numerical stiffness of the system.
Beyond the concept of attractors, center manifolds and inertial
manifolds~\cite{Temam1990} are probably the most prominent examples of such
manifolds.
Another popular concept are normally attracting invariant manifolds (NAIMs),
which are normally hyperbolic invariant manifolds (NHIMs) with empty unstable
bundle.
NHIMs were introduced to generalize the structural
stability~\cite{Fenichel1971,Sakamoto1990,Bates1999,Eldering2012} and the Stable
Manifold Theorem~\cite{Hirsch1970} of hyperbolic fixed points and hyperbolic
periodic orbits to higher dimensions.
These results lie at the core of modern geometric singular perturbation theory
(GSPT) and lead to continuative concepts such as so-called slow invariant
manifolds (SIMs) in the context of slow-fast systems~\cite{Fenichel1979}.
The idea is to construct invariant manifolds using small perturbations and the
structural stability of NHIMs.

There are, however, some critical issues related to these constructions,
especially when considering noncompact critical manifolds.
(The compact case is comprehensibly presented by Hirsch et al.~\cite{Hirsch1977}.)
The two feasible options are to either apply the theorem by
Fenichel~\cite{Fenichel1979} for compact, normally hyperbolic subsets of the
critical manifold with relatively loose assumptions or the persistence theorem
by Eldering for noncompact NAIMs~\cite{Eldering2012}.
The problem with the first approach is the non-uniqueness of the invariant
manifold and the possible loss of normal hyperbolicity, which was extensively
studied by, e.g., Kaneko~\cite{Kaneko1986}, Bonatti et al.~\cite{Bonatti2006},
and Haro and de la Llave~\cite{Haro2006}.
Both issues are addressed by Eldering in the case of NAIMs.
While the uniqueness of the perturbed manifold is a result of the general
setting in bounded geometry, Eldering resolves the breakdown of normal
hyperbolicity -- at least in the case of trivial bundles~\cite{Eldering2013}.
Another result on the persistence of Zhou and Zhang~
Although this second approach is certainly useful for a wide class of
applications the boundedness and uniform continuity of feasible perturbations
is rather restrictive, e.g., in the realm of chemical kinetics.
Furthermore, it is difficult to predict the feasible extent of perturbations
for both approaches.
Numerical algorithms building upon these results -- such as the one by Broer,
Osinga, and Vegter~\cite{Broer1997} -- inherit the above limitations.\vspace{1em}

The mathematical difficulty of capturing an appropriate and convenient concept
for attracting invariant manifolds of slow motion has been predated and
accompanied by more pragmatic computational methods in the areas of science and
engineering.
This is not surprising as the wide range of scientific applications has been a
strong motivation to study abstract dynamical systems in the first place.
In chemical kinetics this development started with the
quasi-steady-state assumption (QSSA)~\cite{Bodenstein1913,Chapman1913} and the
partial equilibrium assumption (PEA)~\cite{Michaelis1913}.
Both methods excel in terms of conceptual simplicity and, thus, are still in
use despite the availability of more sophisticated approaches.

Inspired by the new insights in invariant manifold theory mentioned above and
fuelled by the growing demand for efficient simulations of detailed
high-dimensional models, new dimension reduction methods were proposed,
beginning in the 1980s.
One of these is the computational singular perturbation (CSP) technique by
Lam and Goussis~\cite{Lam1985,Lam1988} with clear nominal reference to GSPT\@.

Later on, among others the intrinsic low-dimensional manifold (ILDM)
method by Maas and Pope~\cite{Maas1992}, the zero-derivative principle (ZDP) by
Gear et al.~\cite{Gear2005}, the flow curvature method
(FCM)~\cite{Ginoux2006}, and the stretching-based diagnostics
(SBD)~\cite{Adrover2007} were introduced.
While these methods indeed provide manifolds, along which the dynamics are slow
and/or attracting in some sense, they are not invariant in general.
Beginning in 2004 this deficit is addressed by trajectory-based methods such as
the trajectory-based optimization approach (TBOA) by
Lebiedz~\cite{Lebiedz2004}, the invariant constrained equilibrium edge
preimage curve (ICE-PIC) method by Ren and Pope~\cite{Ren2005} and some methods
using finite-time Lyapunov exponents (FTLEs), e.g., by Mease et
al.~\cite{Mease2016}.

The current article combines the TBOA with ideas of FTLE in order to construct
a sequence of invariant manifolds that cluster around NAIMs under certain
assumptions.
The main advantage of this approach is the independence of the otherwise
extensively used persistence property of NHIMs.
This makes it a viable option for applications without an obvious slow-fast
formulation or whenever persistence fails.

In contrast to previous articles by Lebiedz et
al.~\cite{Lebiedz2006a,Lebiedz2010,Lebiedz2011,Lebiedz2011a,Lebiedz2013,Lebiedz2014,Lebiedz2016,Heiter2018},
we formulate the TBOA more generally on Riemannian manifolds.
This generalization is useful if nonlinear conservation laws are considered,
e.g., adiabatic conditions in chemical kinetics.

Section~\ref{sec:basics} defines normal hyperbolicity and normal attraction.
Afterwards, we provide sufficient conditions for generalized normally
hyperbolic orbits.
These conditions are formulated in terms of Lyapunov exponents and they do not
require knowledge about the vector bundle in the definition of NAIMs.

Section~\ref{sec:numerics} shows how the TBOA can be used to approximate NAIMs
and applies the method to two two-dimensional reference models and one
realistic combustion model.

In Section~\ref{sec:higher_dim}, the results of the previous two sections are
extended to higher dimensions.
We extend the theoretic results of Section~\ref{sec:basics} and formulate a
numeric method similar to the one in Section~\ref{sec:numerics} to approximate
NAIMs of higher dimensions.

\section{Technical setting and basic concepts}\label{sec:basics}
Let $\ambient$ be a complete, $q$-dimensional Riemannian manifold with metric tensor
$g$.
Denote the space of all smooth vector fields on $\ambient$ by $\mathfrak{X}(\ambient)$.
In this article we study smooth injective semiflows on $\ambient$, i.e., smooth
monoid actions $\flow:\R_{\ge0}\times \ambient \to \ambient, \flow^t(p):=\flow(t,p)$ with
$\flow^t$ being injective for all $t$.

For each $p\in \ambient$ we further define the following connected open set
\begin{equation*}
  \Def{p} :=
  \{s-t \mid \flow^t(\tilde p) = \flow^s(p), \tilde p\in \ambient, t,s\ge0\}.
\end{equation*}
The openness due to the local Lipschitz continuity of the infinitesimal
generator of the smooth action $\flow$.\\
The injectivity enables us to uniquely define $\flow^{-t}(p):=\tilde p$ whenever
$t\in\Def{p}$ and $\flow^t(\tilde p)=p$.

\subsection*{Normally attracting invariant manifolds}
Our goal is to detect and numerically approximate normally attracting invariant
manifolds (NAIMs) in $\ambient$.
These are normally hyperbolic invariant manifolds (NHIMs) with empty unstable
bundle.
Usually, NHIMs are defined analogously to \emph{uniformly} hyperbolic sets.
In this article, however, we work with the following nonuniform definition,
which is related to \emph{nonuniformly} hyperbolic sets.
This less restrictive version simplifies our later constructions.

\begin{definition}[Normally attracting invariant manifold]\label{def:naim}
  Let ${\{\flow^t\}}_{t\ge0}$ be a smooth semiflow on a Riemannian manifold $\ambient$.
  An invariant (immersed) $m$-submanifold $\nhim$ of $\ambient$ is called
  \textbf{(nonuniformly) relatively normally attracting} if there is an invariant smooth
  vector bundle $E$ over $\nhim$ with
  \begin{equation*}
    \T\ambient|_M=\T\nhim \oplus E,
  \end{equation*}
  corresponding projections $\pi_M:\T\ambient\to\T\nhim$, $\pi_E:\T\ambient\to E$,
  and non-negative numbers $0\le\nu_C<\nu$ such that for each $p\in\nhim$
  there is some $C(p)\ge 0$ with
  \begin{equation}\label{eq:norm_attr}
    \left\|\diff\flow^t(w_p)\right\|\|v_p\|
    \le C(p) \e^{-t\nu} \left\|\diff\flow^t(v_p)\right\|\|w_p\|
  \end{equation}
  for all $w_p\in E_p, v_p\in\T_p\nhim$, and $t\ge 0$,
  as well as
  \begin{equation}\label{eq:lipc}
    C(p) \e^{-t\nu_C} \le C(\flow^t(p)) \le C(p) \e^{t\nu_C}
  \end{equation}
  for $t\ge 0$.
\end{definition}

\begin{examples}
  One important example is given by \emph{uniformly absolutely normally
  attracting} invariant manifolds with smooth decomposition
  \begin{equation*}
    \T\ambient|_M=\T\nhim \oplus E
  \end{equation*}
  and non-negative numbers $0\le\nu_M<\nu_E$ such that for each $p\in\nhim$
  there are $C_M,C_E>0$ with
  \begin{align*}
    \left\|\diff \flow^t (v_p)\right\| &\ge C_M\,\e^{-t\nu_M}\,\|v_p\| &
    \text{for all}~v_p\in \T_p\nhim, t\ge 0,\\
    \left\|\diff \flow^t (w_p)\right\| &\le C_E\,\e^{-t\nu_E}\,\|w_p\| &
    \text{for all}~w_p\in E_p, t\ge 0,
  \end{align*}
  Using this definition absolute normal attraction obviously implies relative
  normal attraction, by setting
  \begin{equation*}
    C=\frac{C_E}{C_M} > 0 \quad\text{and}\quad
    \nu := \nu_E - \nu_M > 0.
  \end{equation*}
\end{examples}

Our definition is inspired by the concept of eventual normal hyperbolicity as
stated by Hirsch, Pugh, and Shub in~\cite{Hirsch1977}.
It is also inspired by the theory of nonuniform partial
hyperbolicity~\cite{Barreira2007}, which is formulated for general invariant
sets, considers arbitrary smooth vector bundles, and allows $\nu$ and $\nu_C$
to be nonuniform, (flow-)invariant functions on $\nhim$.\vspace{1em}

There has been some research on generalizing results of
Fenichel~\cite{Fenichel1971,Fenichel1974,Fenichel1977,Fenichel1979} and Hirsch
et al.~\cite{Hirsch1970,Hirsch1977} to noncompact invariant manifolds.
The persistence property was extended by Sakamoto~\cite{Sakamoto1990}, Bates,
Lu, and Zeng~\cite{Bates1999} to noncompact NHIMs manifolds in Euclidean and
Banach spaces, respectively.
Later, Eldering~\cite{Eldering2012} provided a similar result for NAIMs in
bounded geometry.
All of the above results rely on uniformity.
Recently, Zhou and Zhang~\cite{Zhou2021} generalized the persistence result to
nonuniform NHIMs and exponentially bounded perturbations in Banach spaces.

Our method approximates nonuniform NAIMs and, thus, we can assume persistence
whenever the additional assumptions in~\cite{Zhou2021} are satisfied.
From a modeling perspective this is important because inaccurate model
parameters are taken into account.

Before continuing we want to capture the following proposition.
\begin{proposition}\less
  \begin{enumerate}[(i)]
    \item For $p\in\nhim$, consider the function
      \begin{equation*}
        C_p: \Def{p} \to \R^+,\quad t \mapsto C(\flow^t(p)).
      \end{equation*}
      Then, property (\ref{eq:lipc}) of Definition~\ref{def:naim} is equivalent to
      \begin{equation*}
        \sup_{p\in \nhim} \Lip(\log \circ\,C_p) \le \nu_C,
      \end{equation*}
      where $\Lip$ denotes the Lipschitz constant of a function.
      For the rest of the article we call this property the \emph{log-Lipschitz
      continuity} of $C_p$.
    \item For given $\nu,\nu_C$ the pointwise smallest function satisfying both
      (\ref{eq:norm_attr}) and (\ref{eq:lipc}) of Definition~\ref{def:naim} is
      given by
      \begin{equation}
        \hat C(p) := \sup_{t\in\Def{p}} c(\flow^t(p))\,\e^{-|t|\nu_C}
      \end{equation}
      for all $p\in\nhim$, where
      \begin{equation*}
        c(p) :=
        \sup_{s\ge0}\;\sup_{\substack{0\neq v_p\in\T_{p}\nhim\\0\neq w_p\in E_p}}
        \frac{\|\diff\flow^{s}(w_p)\| \|v_p\|}
        {\|\diff\flow^{s}(v_p)\| \|w_p\|} \e^{s\nu}.
      \end{equation*}
    \item For all $p\in\nhim$ and $t\in\Def{p}$ we have
      $1 \le c(p) \le \hat C(p)$.
  \end{enumerate}
\end{proposition}
\begin{proof}
  (i) and (iii) are clear.
  (ii) First, we show that $\hat C$ satisfies the inequalities
  (\ref{eq:norm_attr}) and (\ref{eq:lipc}) of Definition~\ref{def:naim}.

  Without loss of generality, choose $0\neq v_p\in\T_{p}\nhim$,
  $0\neq w_p\in E_p$.
  Since $0\in\Def{p}$, we have
  \begin{equation*}
    \hat C(p) \ge \frac{\|\diff\flow^t(w_p)\| \|v_p\|}
    {\|\diff\flow^t(v_p)\| \|w_p\|} \e^{t\nu} \ge 0
  \end{equation*}
  and, thus,
  \begin{equation*}
    \|\diff\flow^t(w_p)\|\|v_p\|
    \le \hat C(p)\,\e^{t\nu} \|\diff\flow^t(v_p)\| \|w_p\|
  \end{equation*}
  for all $t\ge 0$.

  Next, we want to prove the log-Lipschitz continuity along solution curves.
  Observe that
  \begin{equation*}
    \hat C(\flow^{t}(p))
    = \sup_{s+t\in\Def{p}} c(\flow^{s+t}(p))\,\e^{-|s|\nu_C}
    = \sup_{s\in\Def{p}} c(\flow^s(p))\,\e^{-|s-t|\nu_C}.
  \end{equation*}
  For each $t_1,t_2\in\Def{p}$ we have
  \begin{align*}
    \log \hat C(\flow^{t_1}(p)) -\log \hat C(\flow^{t_2}(p))
    &= \sup_{s_1\in\Def{p}}\inf_{s_2\in\Def{p}}
    (\log c(\flow^{s_1}(p)) +\nu_C|s_1-t_1| -\log c(\flow^{s_2}(p))
    -\nu_C|s_2-t_2|)\\
    &\le \sup_{s_1\in\Def{p}} \nu_C(|s_1-t_1|-|s_1-t_2|)
    \le \nu_C |t_1-t_2|.
  \end{align*}

  Finally, assume there was some function $D$ with $D(t)\ge c(\flow^t(p))$ for
  all $t\in\Def{p}$ and $\Lip(\log\circ D)\le \nu_C$ but
  $D(t^*)<\hat C(\flow^{t^*}(p))$ for some $t^*\in\Def{p}$.
  Then,
  \begin{equation*}
    D(t^*) < \sup_{t\in\Def{p}} c(\flow^t(p))\,\e^{-|t-t^*|\nu_C}
    \le \sup_{t\in\Def{p}} D(t)\,\e^{-|t-t^*|\nu_C},
  \end{equation*}
  which implies
  \begin{equation*}
    \sup_{t\in\Def{p}} (\log D(t) -\log D(t^*) -|t-t^*| \nu_C) > 0
  \end{equation*}
  and, thus, contradicts
  \begin{equation*}
    |\log D(t)-\log D(t^*)| \le \nu_C |t-t^*|.
  \end{equation*}
\end{proof}

\subsection*{Lyapunov exponents}
Normal hyperbolicity is closely linked to Lyapunov exponents, which is why we
introduce a granular notation.
First, we define the \emph{finite-time Lyapunov exponent}
$\lambda:\Def{p}\times\T_p \ambient\to\overline{\R}$ for some $p\in \ambient$ by:
\begin{equation*}
  \lambda^t(v_p) := \lambda(t,v_p) :=
  \begin{cases}
    \frac{1}{|t|} \log \frac{\|\diff\flow^t(v_p)\|}{\|v_p\|}
    & t\neq 0,v_p\neq 0,\\[2mm]
    -\infty & v_p=0\\
    \limsup\limits_{\substack{T\to 0\\T\neq 0}} \lambda^T(v_p) & \text{otherwise.}
  \end{cases}
\end{equation*}
Using this definition, the \emph{forward and backward Lyapunov exponents}
$\lambda^+,\lambda^-:\T\ambient\to\overline{\R}$ are given by
\begin{equation*}
  \lambda^+(v_p) := \limsup_{T\to\sup\Def{p}} \lambda^T(v_p),\quad
  \lambda^-(v_p) := \limsup_{T\to\inf\Def{p}} \lambda^T(v_p).
\end{equation*}

In the theory of Lyapunov regularity the \emph{Lyapunov exponent of the adjoint
variational equation} plays an important role.
We will call it \emph{adjoint Lyapunov exponent} for short.
For finite time horizons and some $p\in \ambient$ the \emph{finite-time adjoint
Lyapunov exponent} $\tilde\lambda:\Def{p}\times\T^*_p \ambient\to\overline{\R}$ is
given by
\begin{equation*}
  \tilde\lambda^t(\omega_p) := \tilde\lambda(t,\omega_p) :=
  \begin{cases}
    \frac{1}{|t|} \log \frac{\|\diff(\flow^{-t})^*\omega_p\|}{\|\omega_p\|}
    & t\neq 0,\omega_p\neq 0,\\[2mm]
    -\infty & \omega_p=0\\
    \limsup\limits_{\substack{T\to0\\T\neq 0}} \tilde\lambda^T(\omega_p) & \text{otherwise,}
  \end{cases}
\end{equation*}
where $\diff(\flow^{-t})^*_{\flow^t(p)}:\T^*_{\flow^t(p)}\ambient \to \T^*_p \ambient$ denotes
the pullback by $\flow^{-t}$ at $\flow^t(p)$.
Again, we define the corresponding \emph{forward} and \emph{backward adjoint
Lyapunov exponents} by
\begin{align*}
  \tilde\lambda^+(\omega_p) = \limsup_{T\to\sup\Def{p}} \tilde\lambda^T(\omega_p),\quad
  \tilde\lambda^-(\omega_p) = \limsup_{T\to\inf\Def{p}} \tilde\lambda^T(\omega_p).
\end{align*}

The theory of Lyapunov exponents provides us with a mature toolkit to study
smooth dynamical systems.
Before we turn towards normal attraction let us recall some of their essential
properties.
According to Bylov et al.~\cite{Bylov1966} (see also~\cite{Barreira2002}) the
three characteristic properties of Lyapunov exponents are:
\begin{enumerate}[(L1)]
  \item $\lambda(\alpha v_p) = \lambda(v_p)$ for all $v_p\in\T_p \ambient$ and
    $\alpha\in\R\setminus\{0\}$;
  \item\label{prop:lyapunov_exp}
    $\lambda(v_p+w_p)\le\max\{\lambda(v_p),\lambda(w_p)\}$ for all
    $v_p,w_p\in \T_p \ambient$;
  \item $\lambda(0_p) = -\infty$.
\end{enumerate}
In our setting of smooth vector flows we have the following additional
properties.
\begin{enumerate}[(L1)]
  \setcounter{enumi}{3}
  \item The forward and backward Lyapunov exponent of global flows are
    invariant.~\cite{Barreira2002}
    In particular, we have
    \begin{equation*}
      \lambda^+|_p = \lambda^+|_{\flow^t(p)}\circ\diff\flow^t|_p
      \quad\text{for all}~t\in\Def{p},~\text{if}~\sup\Def{p}=\infty,
    \end{equation*}
    and
    \begin{equation*}
      \lambda^-|_p = \lambda^-|_{\flow^t(p)}\circ\diff\flow^t|_p
      \quad\text{for all}~t\in\Def{p},~\text{if}~\inf\Def{p}=-\infty.
    \end{equation*}
  \item The forward (backward) Lyapunov exponent and forward (backward)
    disjoint Lyapunov exponents are \emph{dual}~\cite{Barreira2017}, i.e., for
    all dual bases $v_1,\dots,v_q\in\T_p\ambient$,
    $\omega^1,\dots,\omega^q\in\T^*_p\ambient$ and $t\in\bar\R$ we have
    \begin{equation*}
      \lambda^\pm(v_i)+\tilde\lambda^\pm(\omega^i)\ge 0.
    \end{equation*}
\end{enumerate}
The finite-time Lyapunov exponents satisfy
\begin{enumerate}[(L1)]
  \setcounter{enumi}{5}
  \item $\lambda^t(\diff\flow^s(v))
    = \frac{|s+t|}{|t|}\lambda^{s+t}(v) - \frac{|s|}{|t|}\lambda^s(v)$ for all
    $s,t\in\Def{p}$ with $s+t\in\Def{p}$.
\end{enumerate}

Lyapunov exponents are tightly connected to nonuniform (partial) hyperbolicity
as was shown by Pesin~\cite{Pesin1976,Bunimovich2000}.
He established that an invariant set of regular points is nonuniformly
hyperbolic if the Lyapunov exponent on it does not vanish.

The following proposition provides necessary conditions for nonuniform NAIMs in
terms of Lyapunov exponents.
Afterwards, these conditions enable us to generalize the concept of normal
hyperbolicity, which will be important for the rest of the article.

\begin{proposition}
  Let $\flow$ be a smooth flow on a Riemannian manifold $\ambient$ and $\nhim$
  a relatively normally attracting manifold of $\flow$.
  Then, every $p\in\nhim$ satisfies the following inequalities
  \begin{enumerate}[(a)]
    \item $\max \lambda^+[E_p] < \min \lambda^+[\T_p\nhim \setminus \{0_p\}]$ and
    \item $\min \lambda^-[E_p \setminus \{0_p\}] > \max \lambda^-[\T_p\nhim]$.
  \end{enumerate}
  In addition, if $\nhim$ is an absolute NAIM, we have
  \begin{enumerate}[resume*]
    \item $\max \lambda^+[E_p] < 0$ and
    \item $\min \lambda^-[E_p\setminus\{0_p\}] > 0$
  \end{enumerate}
  for all $p\in\nhim$.
\end{proposition}
\begin{proof}
  The inequalities are direct consequences of the definition.
\end{proof}

A continuous invariant vector bundle over an invariant Riemannian manifold is
called \emph{generalized relative NAIM} if it satisfies the first two
inequalities and \emph{generalized absolute NAIM} if it satisfies all four
inequalities.

The following theorem provides necessary conditions for NAIMs.
These inequalities provide upper and lower bounds for Lyapunov exponents and
finite Lyapunov exponents, which is especially useful when the vector bundle
$E$ is not known a priori.

\begin{theorem}\label{thm:naim_le}
  Let $\flow$ be a smooth global flow on a Riemannian manifold $\ambient$ and
  $\nhim\subset\ambient$ an invariant embedded submanifold.
  If $\nhim$ is nonuniformly normally attracting and $p\in\nhim$, then
  \begin{enumerate}[(i)]
    \item for all $0\neq u\in\NS_p\nhim$, every set of vectors $W$ with
      $\spn W=E_p$, and $t\in\Def{p}$ there is always a vector
      $w=w(u)\in W$ with
      \begin{equation*}
        \lambda^t(w) +\tilde\lambda^t(u^\flat)
        \ge \frac{1}{|t|}\log\cos\measuredangle(u,w),
      \end{equation*}
      where $\flat:\T^*\ambient\to\T\ambient$ denotes the musical isomorphism.
      In particular, this implies
      \begin{equation*}
        \lambda^\pm(w) +\tilde\lambda^\pm(u^\flat) \ge 0
      \end{equation*}
      and
      \begin{equation*}
        \max_{0\neq\bar w\in E_p}\lambda^\pm\left(\bar w\right)
        \ge -\min_{\bar u\in\UN_p\nhim} \tilde\lambda^\pm(\bar u^\flat);
      \end{equation*}
    \item for all $0\neq v\in\T_p\nhim$, $0\neq u\in\mathrm{N}_p\nhim$ we have
      \begin{equation*}
        \lambda^+(v) + \tilde\lambda^+(u^\flat) \ge \nu.
      \end{equation*}
      In particular, for each two bases $v_1,\dots,v_m$ of $\T_p\nhim$ and
      $u_{m+1},\dots,u_q$ of $\mathrm{N}_p\nhim$ we have
      \begin{equation*}
        \lambda^+(v_i) + \tilde\lambda^+(u_j^\flat) \ge \nu
      \end{equation*}
      for $i=1,\dots,m$ and $j=m+1,\dots,q$;
    \item for all $0\neq v\in\T_p\nhim$ and $0\neq u\in\NS_p\nhim$ we have
      \begin{equation*}
        \lambda^-(u)-\lambda^-(v) \ge \nu -\nu_C.
      \end{equation*}
      In particular, we have
      \begin{equation*}
        \min_{\bar u\in\UN_p\nhim}\lambda^-(\bar u)
        -\max_{\bar v\in\UT_p\nhim}\lambda^-(\bar v)
        \ge \nu -\nu_C.
      \end{equation*}
  \end{enumerate}
\end{theorem}
\begin{proof}
  \begin{enumerate}[(i)]
    \item Let $W$ be a set of vectors with $\spn W = E_p$, $u\in\NS_p\nhim$ a
      vector with $u\neq 0$, and $t\in\Def{p}$.
      Then, there must be some $w^*\in W$ such that
      $u^\flat(w^*)=\inner{u,w^*}\neq 0$, because otherwise $u=0$.
      This also implies
      \begin{equation*}
        0\neq u^\flat(w^*) = u^\flat\paren*{\diff\flow^{-t}(\diff\flow^t(w^*)}
        = \paren*{\diff\flow^{-t}}^*_{\flow^t(p)}(u^\flat)\paren*{\diff\flow^t(w^*)}
      \end{equation*}
      and, thus,
      \begin{equation*}
        1 = \paren*{\diff\flow^{-t}}^*_{\flow^t(p)}
        \paren*{\frac{u^\flat}{u^\flat(w^*)}} \paren*{\diff\flow^t(w^*)}
        \le \norm*{\paren*{\diff\flow^{-t}}^*_{\flow^t(p)}
          \paren*{\frac{u^\flat}{u^\flat(w^*)}}} \norm*{\diff\flow^t(w^*)}.
      \end{equation*}
      Finally, we get
      \begin{equation*}
        \tilde\lambda^t(u^\flat)+\lambda^t(w^*)
        \ge \frac{1}{|t|}\log\cos\measuredangle(u,w^*).
      \end{equation*}

    \item Let $0\neq v\in\T_p\nhim$.
      Then, inequality (\ref{eq:norm_attr}) implies
      \begin{equation*}
        \log\frac{\|\diff\flow^t(w)\|}{\|w\|}
        \le \log C(p) + \log\frac{\|\diff\flow^t(v)\|}{\|v\|} - \nu t
      \end{equation*}
       for all $0\neq w\in E_p$ and $t\ge 0$,  and, thus,
      \begin{equation*}
        \lambda^+(w) \le \lambda^+(v)-\nu.
      \end{equation*}
      Using (i) we obtain
      \begin{equation*}
        -\tilde\lambda^+(u) \le \lambda^+(v)-\nu
      \end{equation*}
      for all $u\in\NS_p\nhim$.

    \item The definition of NAIMs implies that for all $p\in\nhim$,
      $0\neq v\in\T_p\nhim$, $0\neq w\in E_p$, and $t\ge0$ we have
      \begin{equation*}
        \norm*{\diff\flow^{-t}(v)} \norm{w}
        \le C(\flow^{-t}(p)) \e^{-\nu t} \norm{v} \norm*{\diff\flow^{-t}(w)}
        \le C(p) \e^{(\nu_C-\nu)t} \norm{v} \norm*{\diff\flow^{-t}(w)}.
      \end{equation*}
      Because of the gap
      \begin{equation*}
        \lambda^-(v) \le \nu_C -\nu +\lambda^-(w) < \lambda^-(w)
      \end{equation*}
      in the backward Lyapunov exponents we have
      \begin{equation*}
        \lambda^-(u) -\lambda^-(v) \ge \nu-\nu_C
      \end{equation*}
      for all $u\in\NS_p\nhim$.
  \end{enumerate}
\end{proof}

If the invariant manifold has dimension one the necessary conditions become
even more convenient, as the following corollary shows.
\begin{corollary}
  Let $\nhim$ be a one-dimensional nonuniform relative NAIM of some smooth
  vector field $X\in\mathfrak{X}(\ambient)$.
  Then, for each $p\in\nhim$ with $X_p\neq0$ and each basis $u_2,\dots,u_q$ of
  $\mathrm{N}_p\nhim$ we have
  \begin{equation*}
    \lambda^+(X_p) + \min_{\tilde u\in\UN_p\nhim} \tilde\lambda^+(\tilde u^\flat) \ge \nu
  \end{equation*}
  and
  \begin{equation*}
    \min_{u\in\UN_p\nhim} \lambda^-(u) - \lambda^-(X_p) \ge \nu-\nu_C.
  \end{equation*}
\end{corollary}

Every Lyapunov exponent $\lambda$ on a $\ambdim$-dimensional space attains at
most $\ambdim$ distinct values for nonzero inputs.
The sublevel sets of these values are linear subspaces of the ambient space
and enable us to assign a multiplicity to each value \cite{Barreira2017}.
Counting them by their multiplicities and sorting them, we get a unique tuple
$(\lambda_1',\dots,\lambda_\ambdim')$ of values satisfying
\begin{equation*}
  -\infty \le \lambda_1' \le \cdots \le \lambda_\ambdim'.
\end{equation*}
For two dual Lyapunov exponents $\lambda, \tilde\lambda$ with respective
corresponding tuples $(\lambda_1',\dots,\lambda_\ambdim')$ and
$(\tilde\lambda_1',\dots,\tilde\lambda_\ambdim')$, the
\emph{Perron regularity coefficient} $\sigma_P$ is defined by
\begin{equation*}
  \sigma_P(\lambda,\tilde\lambda)
  := \max\{\lambda_i' + \tilde\lambda_{\ambdim-i+1}':1 \le i \le \ambdim\}.
\end{equation*}

This coefficient enables us to formulate the following sufficient condition for
generalized normally attracting orbits.
It's a useful tool because it doesn't require knowledge about the vector bundle
of NAIMs.
Instead its existence is derived by its effect on the dual Lyapunov exponent of
normal vectors.
\begin{theorem}\label{thm:sufficient_naim_1d}
  Let $\flow$ be a smooth flow with infinitesimal generator
  $X\in\vfs(\ambient)$ on a Riemannian manifold $\ambient$ and $p\in\ambient$ a
  point with $X_p\neq 0$, $\Def{p}=\R$,
  \begin{equation}\label{eq:backward_ineq}
    \lambda^-(X_p) < \min_{\substack{u\in\UT_p\ambient\\u\perp X_p}} \lambda^-(u),
  \end{equation}
  \begin{equation}\label{eq:forward_ineq1}
    \lambda^+ (X_p) > \ambdim\sigma_P(\lambda^+|_p,\tilde\lambda^+|_p)
    - \min_{\substack{u\in\UT_p\ambient\\u\perp X_p}} \tilde\lambda^+(u^\flat),
  \end{equation}
  and
  \begin{equation}\label{eq:forward_ineq2}
    \lambda^+ (X_p) \ge \max \{\lambda^+(v_2),\dots,\lambda^+(v_\ambdim)\}
  \end{equation}
  for an arbitrary basis $v_1:=X_p,v_2,\dots,v_\ambdim$ of $\T_p\ambient$.

  Then, there is a unique linear subspace $E_p\subset\T_p\ambient$,
  $X_p \notin E_p$  of codimension 1 with
  \begin{enumerate}[(a)]
    \item $\lambda^+(X_p) > \max \lambda^+[E_p]$,
    \item $\lambda^-(X_p) < \min \lambda^-[E_p \setminus \{0\}]$.
  \end{enumerate}
  In addition, the union
  \begin{equation*}
    E := \bigcup_{t\in\R} \bigl(\diff\flow^t [E_p] \times \{\flow^t(p)\}\bigr)
  \end{equation*}
  defines an invariant smooth vector subbundle of $\T\ambient|_{\flow^\R(p)}$,
  and the orbit $\flow^\R(p)$ through $p$ is, thus, a generalized relative
  NAIM.
\end{theorem}
\begin{proof}
  \begin{enumerate}[(a)]
    \item Using the same arguments as in the proof of Lemma 2.4.4 in
      \cite{Barreira2017} we can find dual bases $v_1,\dots,v_\ambdim$ of
      $\T_p\ambient$ and $\omega^1,\dots,\omega^\ambdim$ of $\T^*_p\ambient$,
      normal to $\lambda^+$ and $\tilde\lambda^+$, respectively,
      such that $v_1=X_p$ and
      \begin{equation*}
        \lambda^+(v_1) \ge \cdots \ge \lambda^+(v_m).
      \end{equation*}
      We want to show that $\lambda^+(v_2)<\lambda^+(X_p)$, which
      implies that our candidate
      \begin{equation*}
        E_p := \spn \{v_2,\dots,v_\ambdim\}
      \end{equation*}
      satisfies (a).

      Obviously, we have
      $(\omega^2)^\sharp,\dots,(\omega^\ambdim)^\sharp \perp X_p$
      and therefore
      \begin{equation*}
        \tilde\lambda^+(\omega^i)
        \ge \min_{\substack{u\in\UT_p\ambient\\u\perp X_p}}
        \tilde\lambda^+(u^\flat)
      \end{equation*}
      for all $i=2,\dots,\ambdim$.
      In addition,
      \begin{equation*}
        \lambda^+(v_i) + \tilde\lambda^+(\omega^i)
        \le \sum_{k=1}^\ambdim \lambda^+(v_k) + \tilde\lambda^+(\omega^k)
        = \sum_{k=1}^\ambdim
        \lambda^+_{k,\text{mult.}} + \tilde\lambda^+_{k,\text{mult.}}
        \le \ambdim \sigma_P,
      \end{equation*}
      which implies
      \begin{equation*}
        \lambda^+(v_i) \le \ambdim\sigma_P - \tilde\lambda^+(\omega^i)
        \le \ambdim\sigma_P
        - \min_{\substack{u\in\UT_p\ambient\\u\perp X_p}}
        \tilde\lambda^+(u^\flat).
      \end{equation*}
      for all $i=2,\dots,\ambdim$.
      By assumption, this is strictly less than $\lambda^+(X_p)$.

      Apparently, this strict inequality also implies the uniqueness of linear
      subspace $E_p$ and the linear independence from $X_p$, i.e.,
      $X_p \notin E_p$.

    \item Consider the $E_p$ constructed above and assume there was a vector
      $0\neq v^*\in E_p$ with $\lambda^-(X_p) \ge \lambda^-(v^*)$.
      By construction, $X_p$ and $v^*$ are linearly independent and so we can
      find $\alpha,\beta\in\R, \beta\neq0$ with
      \begin{equation*}
        \alpha X_p + \beta v^* \in \UT_p\ambient \cap X_p.
      \end{equation*}
      The corresponding Lyapunov exponent can be bounded by
      \begin{equation*}
        \lambda^-(\alpha X_p+\beta v^*)
        \le \max\{\lambda^-(v^*),\lambda^-(X_p)\}
        = \lambda^-(X_p),
      \end{equation*}
      which contradicts inequality (\ref{eq:backward_ineq}).

    \item Now, we want to derive the generalized relative normal hyperbolicity
      of $\flow^\R(p)$.

      Denote by $\pi_E:=\pi_{\T\ambient}|_E$ the canonical projection map on
      $E$.
      We need to consider two cases.

      First, assume that the orbit through $p$ is non-periodic.
      This implies that for every $v_q\in E$ there is a unique $t\in\R$ with
      $\flow^t(q)=p$.
      By construction, this also implies $\diff\flow^t(v_q)\in E_p$.
      In addition, having a smooth flow the number $t$ depends smoothly on
      $v_q$.
      Therefore,
      \begin{equation*}
        \pi_E(v_q) = (\flow^{-t}\circ\pi_{\T\ambient}\circ\diff\flow^t)(v_q)
        = \flow^{-t}(p)
      \end{equation*}
      is a smooth projection.
      Because $\diff\flow^t$ is a smooth vector bundle isomorphism, we also
      have $E_q \simeq E_p \simeq \R^{m-1}$.
      Therefore, $E$ defines a smooth vector bundle over $\flow^\R(p)$.

      Next, assume the orbit through $p$ is periodic with period $T$.
      This implies that for every $v_q \in E$ there is some unique $t\in[0,T)$
      such that $\pi_E(v_q) = \flow^t(p) = \flow^{kT+t}(p)$ for all
      $k\in\mathbb{Z}$.
      We want to prove that $\diff\flow^T[E_p] = E_p$.
      Apparently, $\diff\flow^T[E_p]$ forms an $\ambdim-1$-dimensional subspace
      of $\T_p\ambient$ because
      $\diff\flow^T|_{\T_p\ambient}:\T_p\ambient\to\T_p\ambient$ defines an
      isomorphism.
      Furthermore, $\lambda^+[\T_p\ambient\setminus E_p]=\{\lambda^+(X_p)\}$
      and $\lambda^+(\diff\flow^T(v_p)) = \lambda^+(v_p) < \lambda^+(X_p)$ for
      all $v_p\in E_p$.
      As a result, $\diff\flow^T[E_p] \subseteq E_p$ and, since
      $\dim E_p = \dim \diff\flow^T[E_p]$, we have $E_p = \diff\flow^T[E_p]$.
      Similar to the non-periodic case the autonomy of the dynamic system
      implies smoothness of $\pi_E$ and
      $\pi_E^{-1}(q) \simeq E_p \simeq \R^{\ambdim-1}$.
      Therefore, $E$ also defines a smooth vector bundle in this case.
  \end{enumerate}
\end{proof}

\section{Review Trajectory-based optimization approach}
This section briefly recapitulates the core ideas of the trajectory-based
optimization approach by Lebiedz et
al.~\cite{Lebiedz2004,Lebiedz2006a,Lebiedz2010,Lebiedz2011,Lebiedz2011a,Lebiedz2013,Lebiedz2014,Lebiedz2016,Heiter2018,Lebiedz2022}.

When it comes to the determination of NHIMs the challenge is to capture both
the geometric and the dynamic structure.
Even for the case of NAIMs this problem remains.
Methods such as ILDM and FCM solve this by focusing on the local geometric
properties of the vector field.
The TBOA, however, tries to capture both, which ensures invariance and
smoothness of the resulting manifold.
This is achieved by applying the flow map to a set of `comparable' points and
locating extremal points or ridges with respect to some geometric measure.
The union of the corresponding trajectories is the prospect (topological)
manifold.
In this paper we focus on extremal points, which correspond to one-dimensional
manifolds inheriting the smoothness of the flow.
For higher dimensional submanifolds extremal ridges should to be considered.

In an abstract way, applying TBOA means solving an optimization problem of
type:
\begin{equation*}
  \max_{p\in K} F_T(p),
\end{equation*}
where $K\subset \ambient$.
Ideally, the subset $K$ intersects with `most' trajectories and contains points
that are in some sense comparable.
The objective function $F_T$ is parameterized by the finite time horizon $T$.

In the following we discuss reasonable choices for $K$ and $F_T$.

Regarding $K$, previous
articles~\cite{Lebiedz2006a,Lebiedz2011a,Lebiedz2013a,Lebiedz2016} introduced
holonomic constraints of the form
\begin{equation*}
  K = \{p\in \ambient \mid \phi^j(p) = \bar p^j\in\R, j\in I_{\text{fixed}}\}
\end{equation*}
for some fixed (global) chart $\phi:\ambient\to\R^q$, values
$\bar p^j\in\R, j\in I_{\text{fixed}}$, and some index set
$I_{\text{fixed}}\subset\{1,\dots,n\}$, which often corresponds to a reasonable
splitting of the coordinates into slow and fast ones.
In the context of chemical reactions one usually refers to reaction progress
variables (RPVs) and non-reaction progress variables (non-RPVs).

The current article suggests another choice for $K$, namely
\begin{equation}\label{eq:new_constraint}
  K = K_\varepsilon := \{p\in \ambient: \|X_p\|=\varepsilon\}
\end{equation}
for $\varepsilon>0$ sufficiently small.
Obviously, in the case of non-autonomous dynamics one has to fix some
additional point in time for $K_\varepsilon$ to be well-defined.
This article, however, focuses on the autonomous case.

By construction, this choice excludes fixed points and compares points based on
the attained velocity of the vector field.
Our choice~(\ref{eq:new_constraint}) can be considered less arbitrary as we do
not rely on choosing an appropriate coordinate chart and slow-fast splitting;
some positive number $\varepsilon$ is sufficient.
The idea is to detach the model reduction from heuristics of the underlying
scientific discipline and rely solely on the mathematical model.

The objective function $F_T$ was primarily based on heuristic considerations
and inspired by physical concepts such as entropy and
energy~\cite{Lebiedz2004}.
It is further discussed
in~\cite{Reinhardt2008,Lebiedz2010,Lebiedz2011,Lebiedz2016}.

This article suggests one of the following objective functions:
\begin{align*}
  F_T(p) &:= f_T(p),
  &F_T(p) &:= \sup_{t\in (0,T)\cap\Def{p}} f_t(p),\\
  F_T(p) &:= \sup_{t\in (-T,0)\cap\Def{p}} f_t(p),
  &\text{or}~%
  F_T(p) &:= \sup_{t\in (-T,T)\cap\Def{p}} f_t(p),
\end{align*}
where $f_t(p)$ could be one of the following:
\begin{enumerate}[(a)]
  \item $\begin{aligned}[t]
      f_t(p) := \lambda^t(X_p)
    \end{aligned}$,
  \item $\begin{aligned}[t]
      f_t(p) := \lambda^t(X_p)
        +\inf_{\substack{w\in\UT_p \ambient\\w\perp X_p}} \tilde\lambda^t(w^\flat)
    \end{aligned}$,
  \item or
    \begin{equation*}
      f_t(p) :=
      \frac{1}{|t|} \sup_{\tau\in(0,T-t)}
      \left(\tau\nu^T(p)
        -\tau\lambda^{-\tau}(\diff\flow^t(X_p))
        -\inf_{\substack{u\in\UT_p\ambient,\\u\perp X_p}}
        \tau\tilde\lambda^\tau\left(\diff(\flow^{-t})^*_{\flow^t(p)}(u^\flat)\right)
      \right) + \nu^T(p),
    \end{equation*}
    where
    $\begin{aligned}[t]
      \nu^T(p) := \lambda^T(X_p)+\min_{\substack{u\in\UT_p\ambient,\\u\perp X_p}} \tilde\lambda^T(u^\flat)
    \end{aligned}$.
\end{enumerate}
The idea is to deduce some inequality that resembles the ones in
Definition~\ref{def:naim}, Theorem~\ref{thm:naim_le}, and
Theorem~\ref{thm:sufficient_naim_1d}.
Later in this article we further elaborate this issue.

\section{Main result}
\begin{figure}
  \centering
  \includegraphics[width=0.8\textwidth]{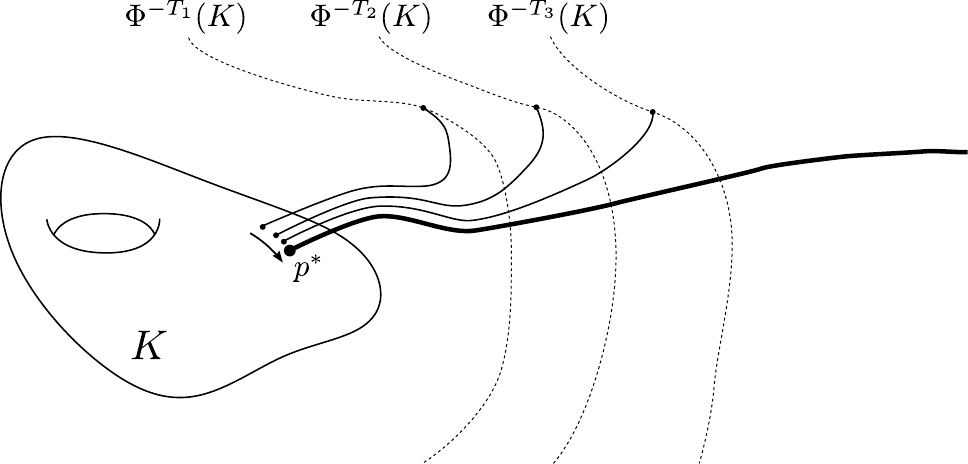}
  \caption{Illustration of the NAIM approximation through TBOA.
    For each time horizon $T_k$ there is a minimizing solution curve with end
    point in $K$.
    The end points of minimizing solution curves then form a net in $K$ over
    time.
    Each normally-attractive orbit intersecting $K$ then corresponds to an
    accumulation point.
  }
\end{figure}
Numerical research provides us with efficient algorithms that approximate
Lyapunov exponents of a specific trajectory
(see~\cite{Benettin1980,Eckmann1985,Sandri1996}).
In our case, however, the goal is to find a trajectory with a minimal Lyapunov
exponent in flow direction.

While the strong algebraic structure of Lyapunov exponents implied some
convenient simplifications regarding the detection of spectral gaps, their
non-continuous nature conflicts with common smooth optimization algorithms.
This section suggests a method to circumvent this inconvenience and find
trajectories with extremal directional Lyapunov exponents.

The following lemma is central to this endeavor.
It states that every accumulation point of a net of minima
corresponding to a family of pointwise monotonously increasing lower
semicontinuous functions on a compactum is a minimum of the pointwise supremum.
This will enable us to calculate smooth finite-time Lyapunov exponents to
approximate the ordinary ones.

\begin{lemma}\label{lemma:top_minima}
  Let $K$ be a compact topological space, $\mathcal{T}\subset\R$ a nonempty
  index set and ${(F_t)}_{t\in \mathcal{T}}$ a family of lower semi-continuous
  functions $F_t:K\to\R\cup\{\infty\}$, pointwise monotonously
  increasing in $t$, i.e.,
  \begin{equation*}
    t_1 \le t_2 \quad\Rightarrow\quad F_{t_1}(p) \le F_{t_2}(p)
    \qquad\text{for all}~p\in K.
  \end{equation*}
  For each $t\in\mathcal{T}$, denote by $\emptyset\neq M_t\subset K$ the set
  of (global) minimum points of $F_t$, and by
  \begin{equation*}
    F(p):=\lim_{t\to\sup\mathcal{T}} F_t(p) = \sup_{t\in\mathcal{T}} F_t(p)
  \end{equation*}
  the lower semi-continuous limit function.
  Each accumulation point of a net ${(x_t)}_{t\in\mathcal{T}}$ with $x_t\in
  M_t$ for all $t$ is a minimum point of $F$.

  In other words, the nonempty outer limit (w.r.t.\ the topology of $K$) of the
  family ${(M_t)}_{t\in\mathcal{T}}$, consists of (global) minima of $F$.
\end{lemma}
\begin{proof}
  It is well-known that the pointwise supremum of any family of lower
  semi-continuous functions is itself lower semi-continuous.
  In addition, each lower semi-continuous function on a compact space attains
  its infimum, which might in general also include $\pm\infty$.

  Now, consider an outer accumulation point $p_*$ of
  ${(M_t)}_{t\in\mathcal{T}}$, that is, for all $s\in \mathcal{T}$ and all
  neighborhoods $U$ of $p_*$ there is a $t\in\mathcal{T}, s \le t$ such that
  $M_t \cap U \neq \emptyset$.

  Choose an arbitrary $\R \ni h < F(p_*)$ (which exists because
  $-\infty<F(p_*)$).
  By definition of $F$ there is an $s\in\mathcal{T}$ such that $h<F_s(p_*)$.
  The lower semi-continuity of $F_s$ in turn implies that there exists a
  neighborhood $V$ of $p_*$ such that $h<F_s(p)$ for all $p\in V$.

  Next, the accumulation property of $p_*$ provides a
  $T\in\mathcal{T}, s\le T$, such that $M_T \cap V \neq \emptyset$, i.e.,
  there is a minimal point $x_T$ of $F_T$ in $V$.
  We obtain the following order:
  \begin{equation*}
    h < F_s(x_T) \le F_T(x_T) \le F_T(p) \le F_t(p) \le F(p)
  \end{equation*}
  for all $p\in K$ and all $t\ge T$.
  Finally, this implies $F(p_*) \le F(p)$ for all $p\in K$.
\end{proof}

Obviously, this result also holds for the maxima of the pointwise infimum of
upper semi-continuous functions if we change signs.

We obtain the following immediate consequence for our specific setting.
\begin{theorem}
  Consider a compact subset $K$ of a Riemannian manifold $\ambient$, a smooth
  injective semiflow $\flow:\R_{\ge0}\times \ambient \to \ambient$, and a family
  ${(f_t)}_{t\in\R}$ of lower semi-continuous functions
  $f_t:\ambient_t:=\{p\in \ambient:t\in\Def{p} \} \to [0,\infty]$.
  Then, by
  \begin{gather*}
    F_T^+ (p) := \sup_{t\in [0,T]\cap\Def{p}} f_t(p),\\
    F_T^- (p) := \sup_{t\in [-T,0]\cap\Def{p}} f_t(p),\\
    F_T^\pm (p) := \sup_{t\in [-T,T]\cap\Def{p}} f_t(p),
  \end{gather*}
  respectively, we define families $(F_T)_T$ of lower semi-continuous
  functions, each defined for all $p\in \ambient$ and pointwise monotonously
  increasing in $T$.
  Then, for all $p\in K$ the functions
  $F^+(p) := \sup\limits_{T\ge0} F_T^+(p)$,
  $F^-(p) := \sup\limits_{T\ge0} F_T^-(p)$,
  and $F^\pm(p) := \sup\limits_{T\ge0} F_T^\pm(p)$,
  respectively, have each at least one minimum point $p_*\in K$ that is also
  an accumulation point of the corresponding net of minimum points.
\end{theorem}
For convenience we drop the $\pm$, $-$, and $+$ in $F_T^\pm$, $F_T^-$ and
$F_T^+$ whenever it is not important to distinguish.
\begin{proof}
  This is a direct consequence of the preceding lemma.
\end{proof}

Most important for our purposes are the following applications.
These enable us to determine normally attracting trajectories by solving
finite-dimensional optimization problems and checking if the attained objective
value is negative.
This is a far-reaching generalization of the model-specific results by Lebiedz,
Siehr, and Unger~\cite{Lebiedz2011a}.
\begin{examples}\label{ex:gs}
  Let $X_p := \deriv{}{t}\flow^t(p)|_{t=0}$ be the
  infinitesimal generator of $\flow$.
  \begin{enumerate}[(a)]
    \item\label{it:gs1} For
      $f_t(p):=\lambda^t(X_p)$
      and the corresponding $F^+(p_*)<\infty$ we obtain the inequality
      \begin{equation*}
        \|\diff\flow^t(X_{p_*})\| \le \e^{tF^+(p_*)} \|X_{p_*}\|
      \end{equation*}
      for every $t\ge0$.
    \item For $f_t(p):=\lambda^t(X_p)$ and the corresponding
      $F^\pm(p_*)<\infty$ we obtain the inequality
      \begin{equation*}
        \|\diff\flow^t(X_{p_*})\|
        \le \e^{|t|F^\pm(p_*)} \|X_{p_*}\|
      \end{equation*}
      for every $t\in\R$.
    \item For $f_t(p) := \lambda^t(X_p)
      -\inf\limits_{\substack{v\in\UT_p \ambient\\v\perp X_p}} \tilde\lambda^t(v^\flat)$
      and the corresponding $F^+(p_*)<\infty$ we obtain the inequality
      \begin{equation*}
        \|v\| \left\|\diff\flow^t(X_{p_*})\right\|
        \le \e^{tF^+(p_*)} \|X_{p_*}\| \left\|\diff\flow^t(v)\right\|
      \end{equation*}
      for every $t\ge 0$ and $v\perp X_{p_*}$, which resembles the defining
      inequality of NAIMs.
  \end{enumerate}
\end{examples}

\section{Numerical aspects}\label{sec:numerics}
Using the results of Example 5 we can formulate an algorithm to approximate
NAIMs of dimension one.\vspace{1em}

Given: $(\ambient,g)$ Riemannian manifold, $X: \ambient \to \T\ambient$ vector
field on $\ambient$, and a number $N\gg 0$.
\begin{enumerate}
  \item Choose a nonempty, compact subset \(K\) that contains no fixed points
    but intersects as many trajectories as possible, e.g., an appropriate
    level set of the vector field length.
  \item Use some smooth numerical optimization algorithm, e.g., the inner
    point method, to approximate a minimum point $p^*\in K$ of the
    corresponding $\mu_N|_K$.
    This also involves using a numerical ODE solver to approximate the
    generated flow.
  \item Approximate the solution curve $\gamma$ through $p^*$ on some
    appropriate time interval $J\subset\R$ using some numerical ODE solver.
  \item Return the approximation of $\gamma$.
\end{enumerate}

For numerical validation we applied our method to three test models.
The first model we tested is the two-dimensional \emph{Davis-Skodje model}.
This model is given by the system
\begin{align*}
  \dot x =& -x\\
  \dot y =& -\gamma y+\frac{(\gamma-1)x+\gamma x^2}{{(1+x)}^2},
\end{align*}
where $\gamma \gg 1$ and $x>-1$.
The corresponding NAIM is given by
\begin{equation*}
  S_\text{DS} = \left\{(x,y)\in\R^2\mid x>-1, y=\frac{x}{1+x}\right\}.
\end{equation*}
Figure~\ref{fig:ds_results} shows how an increasing time horizon of the TBOA
leads to better approximations.
\begin{figure}
  \centering%
  \includegraphics[width=\textwidth]{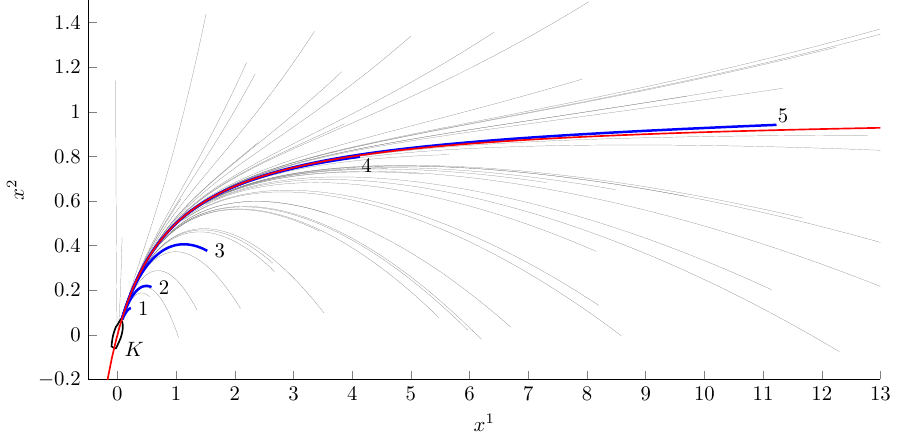}%
  \caption{Results for the Davis-Skodje model with $\gamma=2$ and increasing
    time horizon $T=1,\dots,5$.
    The optimal trajectories (blue) approximate the actual NAIM (red) with
    increasing time horizon.
    Random trajectories (grey) illustrate the bundling behavior of the NAIM.
    The level set $K$ (black) surrounds the equilibrium point at the origin.
    The initial value problems for each step were solved using
    \textsc{Matlab}'s \texttt{ode45} solver.%
  }%
  \label{fig:ds_results}
\end{figure}

Next we tested the \emph{Michaelis-Menten model}, which models an enzymatic
reaction and is given by
\begin{align*}
  \dot x =& -x + xy + (\kappa - \beta) y\\
  \dot y =& \gamma (x + xy + \kappa y).
\end{align*}
In this case we also observe an NAIM similar to the Davis-Skodje model.
Again, the TBOA yields approximations of increasing precision, as presented in Figure~\ref{fig:mm_results}.
\begin{figure}
  \centering
  \includegraphics[width=\textwidth]{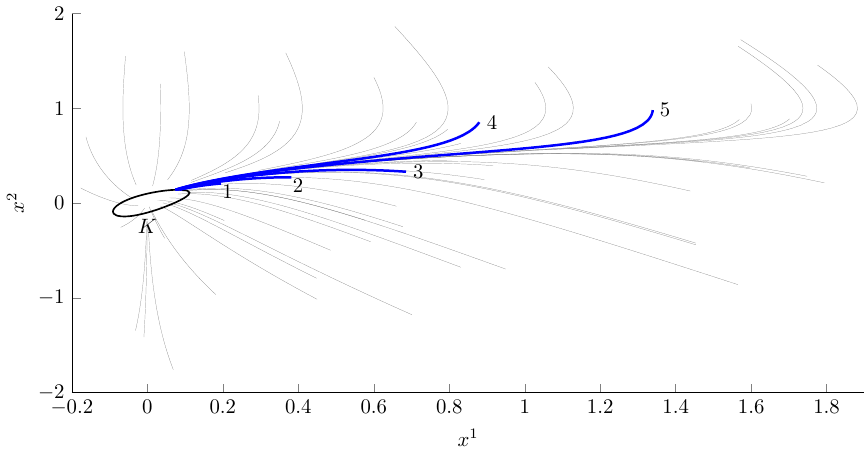}
  \caption{Results for the Michaelis-Menten model with $\gamma=\kappa=\beta$
    and increasing time horizon $T=1,\dots,5$.
    For an increasing time horizon the optimal trajectories (blue) approximate
    the actual NAIM, which is indicated by the bundling of random trajectories
    (grey).
    The level set $K$ (black) surrounds the equilibrium point at the origin.
    The initial value problems for each step were solved using
    \textsc{Matlab}'s \texttt{ode45} solver.%
  }%
  \label{fig:mm_results}
\end{figure}

\subsection*{Hydrogen combustion model}
We also tested our method for a realistic combustion process of physical
chemistry.
The underlying reaction mechanism is given in Table~\ref{tab:icepic_mech}.
This mechanism is a simplified version of the one by Li et al.~\cite{Li2004},
used by Ren et al.~\cite{Ren2006a} and Lebiedz and Siehr~\cite{Lebiedz2013}
to test the ICE-PIC method and a previous version of the TBOA method,
respectively.
\begin{table}%
  \centering
  \begin{tabular}{lllrrr}
    \toprule
  Reaction &&& $A$ / $\unit{mol} \unit{cm}^{-1} \unit{s}^{-1} \unit{K}^{-1}$
  & $b$ & $E_a$  / $\unit{kJ} \unit{mol}^{-1}$ \\
    \midrule
    $\ce{O + H2}    $&$\ce{<=>}$&$\ce{H + OH}$ &$5.08\times10^{04}$& 2.7&26.317 \\
    $\ce{H2 + OH}   $&$\ce{<=>}$&$\ce{H2O + H}$&$2.16\times10^{08}$& 1.5&14.351 \\
    $\ce{O + H2O}   $&$\ce{<=>}$&$\ce{2OH}$    &$2.97\times10^{06}$& 2.0&56.066 \\
    $\ce{H2 + M}    $&$\ce{<=>}$&$\ce{2H + M}$ &$4.58\times10^{19}$&-1.4&436.726\\
    $\ce{O + H + M} $&$\ce{<=>}$&$\ce{OH + M}$ &$4.71\times10^{18}$&-1.0&0.000  \\
    $\ce{H + OH + M}$&$\ce{<=>}$&$\ce{H2O + M}$&$3.80\times10^{22}$&-2.0&0.000  \\
    \bottomrule
  \end{tabular}
  \caption{Simplified hydrogen combustion mechanism as used in~\cite{Ren2006a}.}%
  \label{tab:icepic_mech}
\end{table}

Under isothermal conditions we observe the results presented in
Figure~\ref{fig:icepic_isotherm_results}.

While our test results for this mechanism confirm that the theoretical
findings can also work for realistic kinetic mechanisms there is room for
improvement in terms of robustness and efficiency.

\begin{figure}
  \centering
  \includegraphics[width=\textwidth]{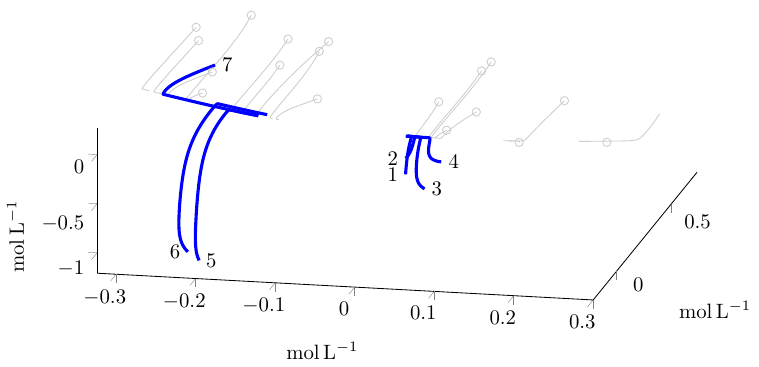}
  \caption{%
    TBOA applied to the isothermal hydrogen combustion mechanism at 3000\,K.
    The observed NAIM is increasingly well approximated for increasing time
    horizons $T=10^{-14}$ to $10^{-13}$.
    In this plot the trajectories are isometrically embedded into the three
    dimensional invariant affine subspace given by the atomic mass constraints.
    While each axis is of dimension $\unit{mol}\unit{L}^{-1}$ the chosen
    coordinates are arbitrary and have no particular physical meaning.
    The initial value problems for each step were solved using
    \textsc{Matlab}'s \texttt{ode15s} solver.%
  }%
  \label{fig:icepic_isotherm_results}
\end{figure}

\section{NAIMs of higher dimensions}\label{sec:higher_dim}
Usually, we are interested in higher-dimensional NAIMs.
To account for those, the TBOA has to be further adapted and extended.
In this section we outline how the same theoretic results can be applied in
this more general case.
The core idea is to replace the finite dimensional optimization of one point
on $K$ with a variational problem that evaluates some objective functional
over some family $\mathcal{K}$ of suitable submanifolds of $K$.
This introduces a fair amount of technicalities that need further elaboration.
In this section we sketch one possible generalization that might work in
some cases.

\begin{figure}
  \centering
  \includegraphics[width=.8\textwidth]{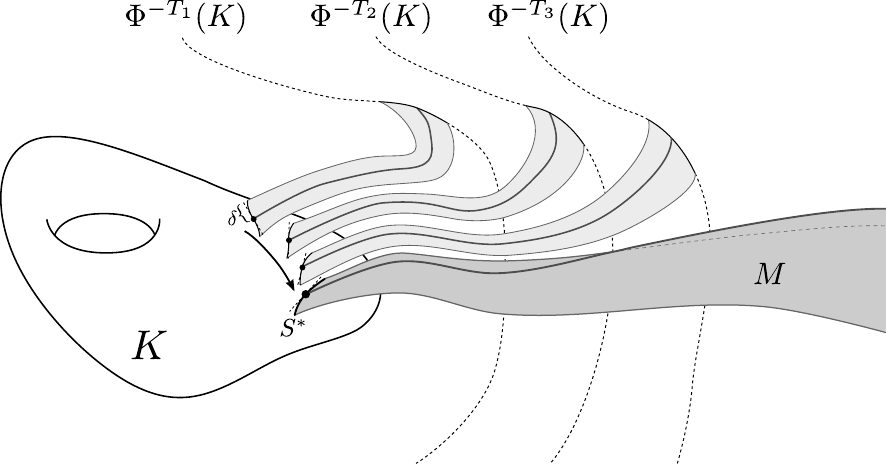}
  \caption{Illustration of the $m$-dimensional NAIM approximation.
    Assume $K$ is compact.
    Each element $D_p$ in the Grassmann bundle $G_{m-1}(\T K)$ corresponds to an
    $m-1$-dimensional totally geodesic submanifold $S$, which is obtained by
    applying the exponential function $\exp|_{\T K}$ to the disk
    $D_p\cap B_\delta$, where $0<\delta\le r_\text{inj}(K)$ is fixed.
    For each time horizon $T_k\ge0$ there is a minimizing element in the
    (compact) Grassmann bundle.
    The totally geodesic submanifolds of the accumulation points in the
    Grassmann bundle are the NAIM candidates.
    In the drawing above, $S^*$ represents one of these candidates, which is
    approached in the Grassmann topology by a net of minimizing submanifolds.
  }%
  \label{fig:naim_approximation_mdim}
\end{figure}
First, we want to discuss possible choices for $\mathcal{K}$.
Due to its vast size the set of all $m-1$-dimensional submanifolds of $K$ is
not suitable as it not only prevents basically all numerical efforts but also
impedes the general solution theory.

Instead, for Riemannian manifolds $K$ with positive injectivity radius
$r_\text{inj}(K)>0$ we suggest using the following subset of totally geodesic
submanifolds of $K$:
\begin{equation*}
  \mathcal{K}_\delta := \big\{\exp[B_\delta(0)\cap D_p] \mid
  D_p \subseteq\T_p K~\text{linear subspace}, \dim(D_p)=m-1, p\in K \big\},
  \qquad 0<\delta\le r,
\end{equation*}
where $\exp: \{v\in\T K:\|v\|<r\} \to K$ denotes the exponential map of the
Riemannian manifold $K$ and $B_\delta(0)$ is the open ball with center $0$
and radius $\delta$.
The construction is illustrated in Figure~\ref{fig:naim_approximation_mdim}.
Compact manifolds are an example for Riemannian manifolds with positive
injectivity radius.
By construction, the elements of $\mathcal{K}_\delta$ form (immersed)
Riemannian submanifolds.

The main advantage of this construction is the finite-dimensional nature of
$\mathcal{K}_\delta$ and the provided construction of each element using
geodesics.
Apparently, $\mathcal{K}_\delta$ is diffeomorphic to the Grassmann bundle
$G_{m-1}(\T K)$ and therefore the variational problem in this case is still
equivalent to a finite dimensional optimization problem on a manifold.
This is a major simplification compared to minimal submanifolds.
Furthermore, this manifold is compact whenever $K$ is, because it is a fiber
bundle over a compact Hausdorff spaces with compact Hausdorff
fibers (cf.~\cite{Steenrod1951}).

When it comes to the choice of the objective functional
Theorem~\ref{thm:naim_le} and Theorem~\ref{thm:sufficient_naim_1d} again
provide us with an idea what to choose.
As for the one-dimensional case each time $T$ provides us with a set
$\Sigma_T$ of submanifolds in $\mathcal{K}_\delta$ minimizing $F_T$.
The relatively simple topology of $\mathcal{K}_\delta$ including compactness
(for compact $K$) enables us to apply Lemma~\ref{lemma:top_minima} again,
which results in at least one $S_*\in\mathcal{K}_\delta$ minimizing
$\lim_{T\to\infty} F_T$.

However, in contrast to the one-dimensional case we have to pay additional
attention to the perpetuation of dimensionality.
In the one-dimensional case this was ensured by simply requiring the
vector field not to vanish on $K$.
For the higher-dimensional case more diligence is necessary.
In particular, the vector field must nowhere be tangent to the minimal
submanifold $S_*$ in $\mathcal{K}_\delta$.
The easiest solution is requiring the vector field to be nowhere tangent to
$K$.

Using one of the $S_*$ we can construct a surface $\nhim$ by considering the
closure with respect to the dynamical system.
It is not clear to us, however, if the strong properties of $S_*$
suffice to assume that $\nhim$ is an (immersed) submanifold.
This is obviously not true for flowouts in general.

However, as the following theorems suggests, self-intersections such as in
Figure~\ref{fig:self_intersection} must be tangent, meaning that the tangent
spaces of all local neighborhoods must align at the intersection points.
The first is a generalization of Theorem \ref{thm:sufficient_naim_1d}.

\begin{theorem}
  Let $\flow$ be a smooth flow with infinitesimal generator
  $X\in\vfs(\ambient)$ on a Riemannian manifold $\ambient$ and
  $S\subset\ambient$ an embedded smooth submanifold where $X$ is nowhere
  tangent and for all $p\in S$ we have $\Def{p}=\R$,
  \begin{equation*}
    \max_{v\in X_p \oplus \T_p S} \lambda^-(v)
    < \min_{\substack{u\in(X_p\oplus\T_p S)^\perp\\u\neq 0}} \lambda^-(u),
  \end{equation*}
  \begin{equation*}
    \min_{\substack{v\in X_p \oplus \T_p S\\v\neq 0}} \lambda^+ (X_p)
    > \ambdim\sigma_P(\lambda^+|_p,\tilde\lambda^+|_p)
    - \min_{\substack{u\in(X_p\oplus\T_p S)^\perp\\u\neq 0}} \tilde\lambda^+(u^\flat),
  \end{equation*}
  and
  \begin{equation*}
    \min\{\lambda^+(X_p),\lambda^+(v_2),\dots,\lambda^+(v_\nhimdim)\}
    \ge \max \{\lambda^+(v_{\nhimdim+1}),\dots,\lambda^+(v_{\ambdim})\}
  \end{equation*}
  for an arbitrary basis $X_p,v_2,\dots,v_\ambdim$ of $\T_p\ambient$ with
  $v_2,\dots,v_\nhimdim\in\T_p S$ and
  $v_{\nhimdim+1},\dots,v_\ambdim\in(X_p\oplus\T_p S)^\perp$.

  Then, there is a unique continuous vector subbundle $E_S$ of $\T\ambient|_S$
  with rank $\dim{\ambdim}-\dim(S)-1$ and
  \begin{enumerate}[(a)]
    \item $\begin{displaystyle}
        \min_{\substack{v\in X_p\oplus\T_p S\\v\neq 0}} \lambda^+(v)
        > \max_{u\in E_p} \lambda^+(u)
      \end{displaystyle}$ for all $p\in S$,

    \item $\begin{displaystyle}
        \max_{v\in X_p\oplus\T_p S}\lambda^-(v)
        < \min_{\substack{u\in E_p\\u\neq 0}} \lambda^-(u)
      \end{displaystyle}$ for all $p\in S$,
  \end{enumerate}
\end{theorem}
\begin{proof}
  The proof is similar to the one of Theorem~\ref{thm:sufficient_naim_1d}.
  First, consider an arbitrary point $p\in S$ and prove the inequalities
  pointwise.
  \begin{enumerate}[(a)]
    \item We choose dual bases $v_1,v_2,\dots,v_\ambdim$ of $\T_p\ambient$ and
      $\omega^1,\dots,\omega^\ambdim$ of $\T^*_p\ambient$, normal to
      $\lambda^+$ and $\tilde\lambda^+$, respectively, such that
      $v_{\ambdim-\nhimdim},\dots,v_\nhimdim\in X_p\oplus\T_pS$, $v_j = X_p$
      for some $j\in\{1,\dots,\nhimdim\}$, and
      \begin{equation*}
        \lambda^+(v_1) \ge \cdots \ge \lambda^+(v_\ambdim).
      \end{equation*}
      Our candidate for the linear subspace is
      \begin{equation*}
        E_p := \spn \{v_{\nhimdim+1},\dots,v_\ambdim\}.
      \end{equation*}
      Because
      $(\omega^{\nhimdim+1})^\sharp,\dots,(\omega^\ambdim)^\sharp \in(X_p\oplus\T_p S)^\perp$
      we have
      \begin{equation*}
        \tilde\lambda^+(\omega^i) \ge \min_{u\in(\substack{X_p\oplus\T_p S})^\perp}
      \end{equation*}
      for all $i=\nhimdim+1,\dots,\ambdim$.
      Using the same arguments as in the proof of
      Theorem~\ref{thm:sufficient_naim_1d} this implies (a).

    \item This works analogous to Theorem~\ref{thm:sufficient_naim_1d} (b).
  \end{enumerate}
  When we sew together the $E_p$ for all $p\in S$ we obtain the (continuous)
  vector bundle
  \begin{equation*}
    E_S := \bigsqcup_{p\in S} E_p
  \end{equation*}
  with the canonical projection map $\pi_{E_S} := \pi_{\T\ambient}|_S$.
\end{proof}

\begin{theorem}
  Let $\flow$ be a smooth flow with infinitesimal generator
  $X\in\vfs(\ambient)$ on a Riemannian manifold $\ambient$,
  $S\subset\ambient$ an embedded smooth submanifold where $X$ is nowhere
  tangent, for all $p\in S$ we have $\Def{p}=\R$, and $E_S$ a smooth vector
  subbundle of $\T\ambient|_S$ with rank $\dim{\ambdim}-\dim(S)-1$ and
  \begin{enumerate}[(a)]
    \item $\begin{displaystyle}
        \min_{\substack{v\in X_p\oplus\T_p S\\v\neq 0}} \lambda^+(v)
        > \max_{u\in E_p} \lambda^+(u)
      \end{displaystyle}$ for all $p\in S$ and

    \item $\begin{displaystyle}
        \max_{v\in X_p\oplus\T_p S}\lambda^-(v)
        < \min_{\substack{u\in E_p\\u\neq 0}} \lambda^-(u)
      \end{displaystyle}$ for all $p\in S$.
  \end{enumerate}

  Then, there are two unique continuous and invariant vector subbundles $E,F$
  of $\T\ambient|_\nhim$ over the image $\nhim$ of the smooth immersion
  $\flow|_{\R\times S}$ with $E|_S = E_S$,
  $F|_S = \T S \oplus \bigsqcup_{p\in S} \spn X_p$, and
  $E\oplus F=\T\ambient|_\nhim$.
\end{theorem}
\begin{proof}
  The two vector bundles are given by
  \begin{equation*}
    E := \bigcup_{t\in\R,p\in S} 
    \left(\diff\flow^t[E_p] \times \{\flow^t(p)\}\right)
  \end{equation*}
  and
  \begin{equation*}
    F := \bigcup_{t\in\R,p\in S} 
    \left(\diff\flow^t[X_p\oplus\T_p S] \times \{\flow^t(p)\}\right),
  \end{equation*}
  respectively.
  The proof is analogous to Theorem~\ref{thm:sufficient_naim_1d}.
  It relies on the fact that for autonomous systems two orbits are identical if
  and only if they share a common point.
  This excludes the possibility that two distinct orbits intersect.
  Self-intersections of $N$ must happen along orbits, which implies that the
  subspaces $E_p$ and $F_p$ of the different folds align at each
  intersection point $p$.
\end{proof}
In general, the image $N$ is not a smooth submanifold.
Therefore, the vector bundles $E$ and $F$ need not be smooth.
Tangential self-intersections are possible.

\begin{figure}
  \centering%
  \includegraphics[width=.6\textwidth]{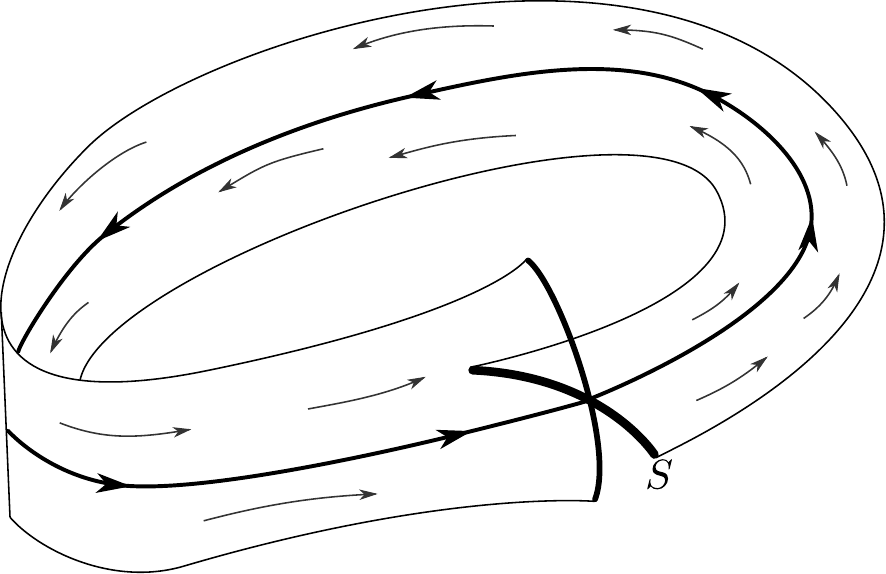}%
  \caption{%
    The smooth flowout of a single point is guaranteed to be an immersed
    submanifold.%
    For higher dimensional embedded submanifolds the situation is captured by
    the Flowout Theorem (see e.g.~\cite{Lee2013}).
    If the vector field is nowhere tangent to an embedded submanifold $S$, then
    $\flow|_{(\R\times S)\cap\mathcal{D}}$ is a smooth immersion, though, not
    necessarily injective.
    This is best illustrated by the flow map
    $\flow^t(\alpha,z)=(\alpha+2\pi t,4t\im z)$ on the vector bundle
    $\ambient:=\sphere^1\times\R^2$ over the unit circle $\sphere^1$, where we
    use the canonical identification $\R^2\cong\CC$ for convenience.
    In our applications this limitation can often be neglected if we assume
    that $S$ is in the basin of attraction of some attractor $\Omega$ and
    $S\cap\Omega=\emptyset$.
    Then, the flowout of $S$ is indeed a immersed manifold because it does not
    contain any periodic orbit.%
  }%
  \label{fig:self_intersection}
\end{figure}

The last consideration we want to mention with respect to the
higher-dimensional case is the choice of geometry on $K$.
Implicitly we used the induced metric, which is up for debate, especially when
considering that the family of totally geodesic submanifolds is small and the
NAIM might be curved with respect to the induced metric as illustrated in
Figure~\ref{fig:curved_naim}.
\begin{figure}
  \centering
  \includegraphics[width=0.6\textwidth]{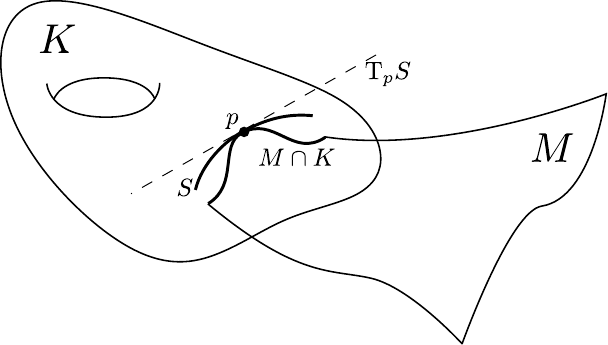}
  \caption{The induced Riemannian metric on $K$ might not be the right choice
    in order to find an NAIM.
    In this figure $\nhim$ denotes an NAIM, which transversally intersects a
    compact embedded submanifold $K$.
    We denote by $S=\{\exp[B_\delta(0)\cap\T_p S]\}\in\mathcal{K}_\delta$
    the unique totally geodesic submanifold of radius $\delta$ with respect to
    the induced Riemannian metric for some $p\in\nhim\cap K$ and assume
    $\T_p S=\T_p(\nhim\cap K)$.%
  }
  \label{fig:curved_naim}
\end{figure}
For this reason one could use a pullback metric like $g^T:=(\flow^T){}^*g$.
However, in order to apply Lemma~\ref{lemma:top_minima} this metric either
has to be fixed a priori or another supremum over a complete family of
metrics has to be introduced in the formula of the objective functional
$F_T$.
Otherwise the pointwise monotonicity of $F_T$ in $T$ is not guaranteed.
Both considerations go beyond the scope of the current article.

\section{Summary \& Conclusion}
The aim of this work is to study the effectiveness of the trajectory-based
optimization approach by Lebiedz~\cite{Lebiedz2004}.
So far this approach was based on physical intuition and geometric heuristics
in order to approximate invariant manifolds of slow motion.
In this article we show how it can be used to approximate normally attracting
orbits.
The method is formulated in a coordinate-free way on abstract Riemannian
manifolds.

In Section~\ref{sec:numerics} we test the theoretic results for two simple
nonlinear models with globally attracting equilibrium points and known normally
attracting orbits, which are frequently used for benchmarks in the field of
manifold-based model reduction in reactive flows.
It was also tested on a realistic isotherm hydrogen combustion mechanism used
by Ren et al.~\cite{Ren2006a} and Lebiedz and Siehr~\cite{Lebiedz2013a}.
The results confirm the theoretical findings.

In Section~\ref{sec:higher_dim} we outline how to apply the method to
submanifolds of higher dimension.
The idea is to consider the exponential map on the Grassmann bundle of an
appropriate submanifold.
There are, however, some open questions in this regard, e.g., which Riemannian
metric should be used on the submanifold.
In addition, further research is necessary to check if this higher-dimensional
algorithm can be implemented efficiently.

\subsection*{Acknowledgements}
The authors acknowledge the financial support from the Klaus Tschira Stiftung
in the project 00.003.2019.
Further, the authors thank Marius Müller and Johannes Poppe for discussions on
the topic.

\printbibliography

\end{document}